\documentclass[12pt,a4paper, reqno]{amsart}
\usepackage{amsmath, amssymb, amsthm, eufrak, setspace, verbatim, xypic, enumitem}
\usepackage[top=2.5cm, bottom=2cm, left=2cm, right=2cm]{geometry}
\usepackage[all]{xy}
\usepackage[pagewise]{lineno}

\theoremstyle{definition}
\newtheorem{theorem}{Theorem}[section]
\newtheorem{definition}[theorem]{Definition}
\newtheorem{proposition}[theorem]{Proposition}
\newtheorem{lemma}[theorem]{Lemma}

\newtheorem{example}[theorem]{Example}

\newcommand{\Z}{\mathbb{Z}}
\renewcommand{\bar}{\overline}

\DeclareMathOperator{\End}{\mbox{End}}
\DeclareMathOperator{\Perm}{\mbox{Perm}}

\DeclareMathOperator{\Aut}{\mbox{Aut}}

\def \O {\mathfrak{O}}
\def \A {\mathfrak{A}}

\begin{document}

\title[Semidirect products in Hopf-Galois theory]{On some semidirect products of skew braces arising in Hopf-Galois theory}

\author{Paul J. Truman}
\address{School of Computer Science and Mathematics \\ Keele University \\ Staffordshire \\ ST5 5BG \\ UK}
\email{P.J.Truman@Keele.ac.uk}



\begin{abstract}
We classify skew braces that are the semidirect product of an ideal and a left ideal. As a consequence, given a Galois extension of fields $ L/K $ whose Galois group is the semidirect product of a normal subgroup $ A $ and a subgroup $ B $, we classify the Hopf-Galois structures on $ L/K $ that realize $ L^{A} $ via a normal Hopf subalgebra and $ L^{B} $ via a Hopf subalgebra. We show that the Hopf algebra giving such a Hopf-Galois structure is the smash product of these Hopf subalgebras, and use this description to study generalized normal basis generators and questions of integral module structure in extensions of local fields. 
\end{abstract}

\maketitle

\section{Introduction} \label{sec_Introduction}

Let $ L/K $ be a finite extension of fields. We say that a $ K $-Hopf algebra $ H $ gives a \textit{Hopf-Galois structure} on $ L/K $ to mean that $ L $ is an $ H $-module algebra and the natural $ K $-linear map $ L \otimes_{K} H \rightarrow \End_{K}(L) $ is an isomorphism. First introduced in \cite{CS69}, Hopf-Galois structures can be used to study inseparable or non-normal field extensions, but there is also considerable interest in their applications to extensions that are Galois in the classical sense, which are the focus of this paper.

One of the accomplishments of Hopf-Galois theory is a partial analogue of the classical Galois correspondence \cite{CS69} \cite[Chapter 7]{HAGMT}. If $ H $ gives a Hopf-Galois structure on $ L/K $ then for each Hopf subalgebra $ H' $ of $ H $ the set
\[ L^{H'} = \{ x \in L \mid h(x) = \varepsilon(h)x \mbox{ for all } h \in H' \} \]
(where $ \varepsilon $ denotes the counit of $ H $) is an intermediate field of the extension $ L/K $. This correspondence between Hopf subalgebras and fixed fields is injective and inclusion reversing, but in general not surjective: we say that an intermediate field $ L' $ of $ L/K $ is \textit{realised} by $ H $ if $ L' = L^{H'} $ for some Hopf subalgebra of $ H $. In this case we have $ [L:L'] = \dim_{K}(H') $, and the $ L' $-Hopf algebra $ L' \otimes_{K} H' $ gives a Hopf-Galois structure on $ L/L' $. Furthermore, if $ H' $ is a \textit{normal} Hopf subalgebra of $ H $ \cite[Section 3.3]{Mon93} then we can construct a quotient $ K $-Hopf algebra $ \bar{H} = H//H' $, which gives a Hopf-Galois structure on $ L'/K $ \cite[Lemma 4.1]{By02} and fits into a short exact sequence of $ K $-Hopf algebras
\begin{equation} \label{eqn_short_exact_sequence_HGS}
K \rightarrow H' \rightarrow H \rightarrow \bar{H} \rightarrow K. 
\end{equation}

It is natural to study the corresponding \textit{extension problem for Hopf-Galois structures}: given a finite extension of fields $ L/K $, an intermediate field $ L' $, and $ K $-Hopf algebras $ H' $ and $ \bar{H} $ such that $ L' \otimes_{K} H' $ gives a Hopf-Galois structure on $ L/L' $ and $ \bar{H} $ gives a Hopf-Galois structure on $ L'/K $, can we construct, or classify, $ K $-Hopf algebras $ H $ that fit into a short exact sequence as in \eqref{eqn_short_exact_sequence_HGS} and give Hopf-Galois structures on $ L/K $ that yield the given Hopf-Galois structures on $ L/L' $ and $ L'/K $? 

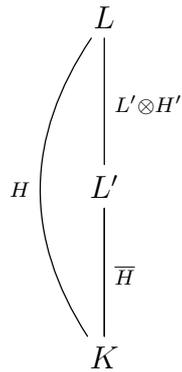
\[
\xymatrixcolsep{3pc} 
\xymatrixrowsep{4pc}
\xymatrix{
 \ar@/_2pc/@{-}[dd]_{H} \ar@{-}[d]^{L' \otimes H'}  L   \\
 \ar@{-}[d] L' \ar@{-}[d]^{\bar{H}} \\
  K  \\
} 
\]

In the case in which $ L/K $ is a Galois extension, the groundbreaking work of Greither and Pareigis \cite{GP87} classifies the Hopf-Galois structures on $ L/K $ (and on $ L/L' $ and the potentially non-normal extension $ L'/K $) in group theoretic terms. Using this framework Crespo, Rio, and Vela develop the theory of \textit{induced} Hopf-Galois structures on split Galois extensions \cite{CRV16a}; this is the only existing work on the extension problem for Hopf-Galois structures on separable extensions. 

The Greither-Pareigis classification reveals a connection between Hopf-Galois structures on Galois extensions and the theory of \textit{skew braces}: these are triples $ (G,\cdot,\circ) $ in which $ (G,\cdot) $ and $ (G,\circ) $ are groups and the binary operations $ \cdot $ and $ \circ $ satisfy a certain compatibility relation \cite{GV17}. This connection is developed by several authors \cite{SV18}, \cite{Ch18}, \cite{Ch19b} and made most precise by Stefanello and Trappeniers \cite{ST23}: they show that if $ L/K $ is a Galois extension with Galois group $ (G,\circ) $ then there is a bijection between Hopf-Galois structures on $ L/K $ and further binary operations $ \cdot $ on $ G $ such that $ (G,\cdot,\circ) $ is a skew brace.

If $ H $ gives a Hopf-Galois structure on $ L/K $, with corresponding skew brace $ (G,\cdot,\circ) $, then the intermediate fields $ L' $ that are Galois over $ K $ and realised by normal Hopf subalgebras of $ H $ correspond with \textit{ideals} $ (A,\cdot,\circ) $ of $ (G,\cdot,\circ) $ (kernels of skew brace homomorphisms). Hence the extension problem for Hopf-Galois structures on the tower of Galois extensions $ L/L' $ and $ L'/K $ corresponds to the extension problem for skew braces
\begin{equation} \label{eqn_short_exact_sequence_SB}
1 \rightarrow (A,\cdot,\circ) \rightarrow (G,\cdot,\circ) \rightarrow (G/A,\cdot,\circ) \rightarrow 1. 
\end{equation}
As we would expect, this extension problem is very difficult in general; however,  various notions of direct and semidirect products of skew braces have been formulated. In this work we use these notions to classify families of extensions of Hopf-Galois structures on split Galois extensions. 

More precisely: in Section \ref{sec_semidirect} we classify skew braces $ (G,\cdot,\circ) $ that are the semidirect product of an ideal $ (A,\cdot,\circ) $ and a \textit{left ideal} $ (B,\cdot,\circ) $. This is equivalent to classifying Hopf-Galois structures on a Galois extension $ L/K $ with Galois group $ (G,\circ) \cong (A,\circ) \rtimes (B,\circ) $ that realise $ L^{(A,\circ)} $ via a normal Hopf subalgebra and realise $ L^{(B,\circ)} $ via a Hopf subalgebra. Since the extension $ L^{(A,\circ)}/K $ is also Galois with Galois group $ (B,\circ) $, we obtain solutions to the extension problem for Hopf-Galois structures on the tower of Galois extensions $ L/L^{(A,\circ)} $ and $ L^{(A,\circ)}/K $; we explore this perspective in Section \ref{sec_examples}. 

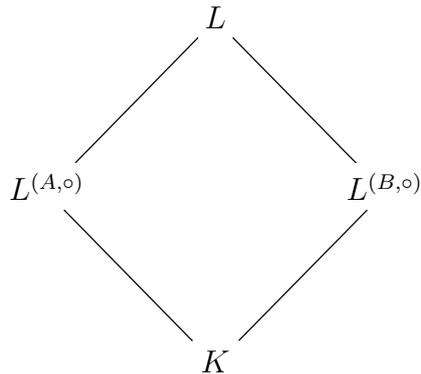
\[
\xymatrixcolsep{3pc} 
\xymatrixrowsep{4pc}
\xymatrix{
& \ar@{-}[dl] L \ar@{-}[dr] & \\
 L^{(A,\circ)} \ar@{-}[dr] & & \ar@{-}[dl] L^{(B,\circ)} \\
& K & \\
} 
\]

We identify several existing classification results as instances of our construction; in particular, in Section \ref{sec_induced_HGS} we interpret the theory of induced Hopf-Galois structures \cite{CRV16a} in the context of skew braces. Finally, in Section \ref{sec_smash} we show that the Hopf algebras that arise via our construction are isomorphic to \textit{smash products} of certain Hopf subalgebras, and show how this description can be applied to address questions concerning integral module structure in extensions of local or global fields.

\section{Skew braces and Hopf-Galois structures on Galois extensions} \label{sec_skew_braces}

A \textit{(left) skew brace} is a triple $ (G,\cdot,\circ) $ in which $ (G,\cdot) $ and $ (G,\circ) $ are groups and 

\begin{equation} \label{eqn_skew_brace}
x \circ (y \cdot z) = (x \circ y) \cdot x^{-1} \cdot (x \circ z) \mbox{ for all } x,y,z \in G, 
\end{equation}

\noindent where $ x^{-1} $ denotes the inverse of $ x $ with respect to $ \cdot $. This inverse need not coincide with the inverse of $ x $ with respect to $ \circ $, which we denote by $ \bar{x} $. The identity elements with respect to $ \cdot $ and $ \circ $ do coincide; we denote this common identity by $ e $. We suppress the notation $ \cdot $ wherever possible. 

To further ease notation we shall often denote a skew brace $ (G,\cdot,\circ) $ simply by $ G $, writing $ (G,\cdot) $ or $ (G,\circ) $ when we wish to specify one of the group structures. Similarly, we shall write $ \Aut(G) $ for the group of skew brace automorphisms of $ G $ (that is: bijections on $ G $ that respect both $ \cdot $ and $ \circ $); this is a subgroup of each of $ \Aut(G,\cdot) $ and $ \Aut(G,\circ) $.

If $ G $ is a skew brace then there is a homomorphism $ \gamma : (G,\circ) \rightarrow \Aut(G,\cdot) $ defined by $ \gamma_{x}(y) = x^{-1}(x \circ y) $ for all $ x,y \in G $; we call this the \textit{$ \gamma $-function} of the skew brace. 

A subset $ A $ of $ G $ is called a \textit{left ideal} of $ G $ if it is a subgroup with respect to either operation and satisfies $ \gamma_{x}(A)=A $ for all $ x \in G $; it then follows that $ A $ is a subgroup with respect to both operations. A left ideal $ A $ of $ G $ is called a \textit{strong left ideal} if 
$ (A,\cdot) \trianglelefteq (G,\cdot) $, and an \textit{ideal} if in addition $ (A,\circ) \trianglelefteq (G,\circ) $; if $ A $ is an ideal of $ G $ then we naturally obtain a quotient skew brace $ (G/A, \cdot, \circ) $. 

Now we describe the correspondence between finite skew braces and Hopf-Galois structures on Galois extensions, following Stefanello and Trappeniers \cite{ST23}. Let $ L/K $ be a Galois extension with Galois group $ (G,\circ) $. By \cite[Theorem 3.1]{ST23} there is a bijection between Hopf-Galois structures on $ L/K $ and binary operations $ \cdot $ on the set $ G $ such that $ (G,\cdot,\circ) $ is a skew brace. By viewing distinct binary operations as giving distinct skew braces, we may interpret this result as a bijection between Hopf-Galois structures on $ L/K $ and skew braces $ (G,\cdot,\circ) $. We shall say that the Hopf-Galois structure corresponding to such a skew brace has \textit{type} $ (G,\cdot) $. The Hopf algebra giving this Hopf-Galois structure is obtained via Galois descent (see \cite[Section 2.1]{ST23}): the group algebra $ L[G,\cdot] $ is an $ L $-Hopf algebra on which the group $ (G,\circ) $ acts semilinearly by the rule
\begin{equation} \label{eqn_skew_brace_HGS_action}
x \star \sum_{y \in G} c_{y} y = \sum_{y \in G} x(c_{y}) \gamma_{x}(y); 
\end{equation}
since this action is compatible with the Hopf algebra structure maps on $ L[G,\cdot] $ the fixed ring $ H = L[G,\cdot]^{(G,\circ)} $ is a $ K $-Hopf algebra, which acts on $ L $ via the rule
\begin{equation} \label{eqn_skew_brace_HGS_action_on_L}
\left( \sum_{y \in G} c_{y} y \right) [t] = \sum_{y \in G} c_{y} y[t]
\end{equation}
to give a Hopf-Galois structure on $ L/K $. 

Under this correspondence the left ideals of $ G $ correspond with Hopf subalgebras of $ H $ (and hence to intermediate fields realised by $ H $) via the rule $ A \mapsto H_{A} = L[A,\cdot]^{(G,\circ)} $. Since a left ideal $ A $ is also a subgroup of $ (G,\circ) $, we have two apparently different ways of obtaining an intermediate field from $ A $: we can form the fixed field $ L^{(A,\circ)} $ via Galois theory, or the fixed field $ L^{H_{A}} $ of the Hopf subalgebra $ H_{A} $. A beautiful facet of this theory is that these fixed fields coincide, so that we can write $ L^{A} $ without ambiguity. A Hopf subalgebra $ H_{A} $ is normal if and only if $ A $ is a strong left ideal of $ G $. In particular, if $ A $ is an ideal of $ G $ then $ H//H_{A} $ gives a Hopf-Galois structure on the Galois extension $ L^{A}/K $; as we would hope, this Hopf-Galois structure corresponds to the quotient skew brace $ G/A $. 

Finally, we recall that each skew brace $ G=(G,\cdot,\circ) $ has an \textit{opposite} skew brace $ G^{op}=(G,\cdot^{op},\circ) $ \cite{KT20}; the Hopf-Galois structures on $ L/K $ corresponding to $ G $ and $ G^{op} $ are also called opposites of one another, and coincide if and only if $ (G,\cdot) $ is abelian.

\section{Some semidirect products of skew braces} \label{sec_semidirect}

We recall from \cite[Proposition 2.2]{FP24} that a skew brace $ G=(G,\cdot,\circ) $ is said to be the \textit{internal semidirect product} of an ideal $ A $ and a sub skew brace $ B $ to mean that
\begin{itemize} 
\item $ A \cap B = \{ e \} $;
\item $ A \cdot B = G $;
\item $ A \circ B = G $.
\end{itemize}

The first of these conditions, together with either of the other two, implies the third. As we would expect, there is a corresponding notion of external semidirect product of skew braces $ A $ and $ B $; these are classified in \cite[Proposition 4.2]{FP24} requiring intricate relations amongst the $ \gamma $-functions of these skew braces and various actions of one on the other. We shall classify skew braces that are the internal semidirect product of an ideal and a \textit{left ideal}; these hypotheses simplify the construction of the corresponding external semidirect products, whilst still allowing for a range of behaviour. 

In order to study the corresponding notion of external semidirect product, we establish some notation. 

\begin{definition} \label{defn_external_SDP}
Given skew braces $ A=(A,\cdot,\circ) $ and $ B=(B,\cdot,\circ) $ and group homomorphisms
\begin{eqnarray*}
\varphi : (B,\circ) \rightarrow \Aut(A,\circ) \\
\theta : (B,\cdot) \rightarrow \Aut(A,\cdot),
\end{eqnarray*}
let $ A \rtimes_{\theta}^{\varphi} B $ denote the Cartesian product $ A \times B $ together with the operations
\begin{eqnarray*}
(a,b) \circ (a',b') = (a \circ \varphi_{b}(a'), b \circ b') \\
(a,b) \cdot (a',b') = (a \cdot \theta_{b}(a'), b \cdot b').
\end{eqnarray*}
(Here $ \varphi_{b} = \varphi(b) \in \Aut(A,\circ) $, and similarly for $ \theta_{b} $.)
\end{definition}

In general $ A \rtimes_{\theta}^{\varphi} B $ need not be a skew brace (we obtain a criterion in Theorem \ref{thm_external_semidirect}). However, we find that if it is a skew brace then it is an internal semidirect product. 

\begin{proposition} \label{prop_external_SDP_is_internal_SDP}
Suppose that $ A \rtimes_{\theta}^{\varphi} B $ is a skew brace. Then $ (A,e) $ is an ideal, $ (e,B) $ is a left ideal, and $ A \rtimes_{\theta}^{\varphi} B $ is the internal semidirect product of these. 
\end{proposition}
\begin{proof}
It is clear that, as groups, $ A \rtimes_{\theta}^{\varphi} B $ is the semidirect product of the normal subgroup $ (A,e) $ and the subgroup $ (e,B) $ with respect to each of the operations. We must show that each of these sets is closed under the $ \gamma $-function of $ A \rtimes_{\theta}^{\varphi} B $. This defined by
\begin{eqnarray*}
\gamma_{(a,b)}(c,d) & = & (a,b)^{-1} ( (a,b) \circ (c,d) ) \\
& = & (\theta_{b}^{-1}(a^{-1}), b^{-1}) (a \circ \varphi_{b}(c), b \circ d) \\
& = & (\theta_{b}^{-1}(a^{-1}(a \circ \varphi_{b}(c)), b^{-1}(b \circ d)) \\
& = & (\theta_{b}^{-1}(\gamma_{a}(\varphi_{b}(c))), \gamma_{b}(d)). 
\end{eqnarray*}
We see quickly that $ \gamma_{(a,b)}(c,e) \in (A,e) $ and that $ \gamma_{(a,b)}(e,d) \in (e,B) $, which completes the proof. 
\end{proof}

Conversely, each skew brace that is the internal semidirect product of an ideal and a left ideal is isomorphic to a suitable external semidirect product as in Definition \ref{defn_external_SDP}. 

\begin{proposition}
Suppose that $ G = (G,\cdot,\circ) $ is a skew brace which is the internal semidirect product of an ideal $ A $ and a left ideal $ B $. Define 
\[ \varphi : (B,\circ) \rightarrow \Aut(A,\circ) \mbox{ by } \varphi_{b}(a) = b \circ a \circ \bar{b} \]
and
\[ \theta : (B,\cdot) \rightarrow \Aut(A,\cdot) \mbox{ by } \theta_{b}(a) = b \cdot a \cdot b^{-1}. \]
Then $ A \rtimes_{\theta}^{\varphi} B $ is a skew brace and the map $ \delta : G \rightarrow A \rtimes_{\theta}^{\varphi} B $ defined by $ \delta(a \circ b) = (a,b) $ is an isomorphism.
\end{proposition}
\begin{proof}
First we establish the useful fact that $ \gamma_{a}(b) = b $ for all $ a \in A $ and $ b \in B $. Consider the element $ b^{-1}\gamma_{a}(b) $. Since $ B $ is a left ideal of $ G $ we have $ \gamma_{a}(b) \in B $ and so $ b^{-1}\gamma_{a}(b) \in B $. On the other hand we have
\begin{eqnarray*}
b^{-1}\gamma_{a}(b) & = & b^{-1}a^{-1}(a \circ b) \\
& = & \theta_{b}^{-1}(a^{-1}) b^{-1} ( b \circ \varphi_{b}^{-1}(a)) \\
& = & \theta_{b}^{-1}(a^{-1}) \gamma_{b}(\varphi_{b}^{-1}(a)). 
\end{eqnarray*}
Using the assumption that $ A $ is an ideal of $ G $ we find that $ \theta_{b}^{-1}(a^{-1}) \in A $ and $ \varphi_{b}^{-1}(a) \in A $, then that $ \gamma_{b}(\varphi_{b}^{-1}(a)) \in A $, and finally that $ b^{-1} \gamma_{a}(b)  \in A $. Hence $ b^{-1} \gamma_{a}(b) \in A \cap B = \{ e \} $, and so $ \gamma_{a}(b)= b $. 

Now we use the standard facts that the binary operation defined on the Cartesian product $ A \times B $ by 
\[ (a,b)\circ (a',b') = (a \circ \varphi_{b}(a), b \circ b') \]
makes $ A \times B $ into a group and that the map $ \delta : (G,\circ) \rightarrow (A \times B,\circ) $ given by $ \delta(a \circ b) = (a,b) $ is an isomorphism of groups. Using this isomorphism we define a second binary operation on $ A \times B $ by 
\[ (a,b)\cdot (a',b') = \delta ( \delta^{-1}(a,b) \cdot \delta^{-1}(a',b') ); \] 
then $ (A \times B, \cdot, \circ) $ is a skew brace isomorphic to $ G $ via $ \delta $. Now we have
\begin{eqnarray*}
(a,b) \cdot (a',b') & = & \delta( (a \circ b) \cdot (a' \circ b') ) \\
& = & \delta( a\gamma_{a}(b)a'\gamma_{a'}(b')) \\
& = & \delta( a b a' b' ) \\
& = & \delta( a \theta_{b}(a') bb' ) \\
& = & \delta( a \theta_{b}(a') \gamma_{a\theta_{b}(a')}(bb') ) \\
& = & \delta (a \theta_{b}(a') \circ bb') \\
& = & (a \theta_{b}(a'),bb').
\end{eqnarray*}
Hence the operations we obtain on $ A \times B $ coincide with those definition of $ A \rtimes_{\theta}^{\varphi} B $, and so $ \delta $ is a skew brace isomorphism from $ G $ to  $ A \rtimes_{\theta}^{\varphi} B $.
\end{proof}

Finally, we classify skew braces of the form $ A \rtimes_{\theta}^{\varphi} B $. Recall that we say that a skew brace $ B $ \textit{acts on} a skew brace $ A $ to mean that there is a homomorphism $ (B,\circ) \rightarrow \Aut(A) $, and that the binary operations on $ A \rtimes_{\theta}^{\varphi} B $ are as in Definition \ref{defn_external_SDP}.

\begin{theorem} \label{thm_external_semidirect}
Let $ A=(A,\cdot,\circ) $ and $ B=(B,\cdot,\circ) $ be skew braces and let
\begin{eqnarray*}
\varphi : (B,\circ) \rightarrow \Aut(A,\circ) \\
\theta : (B,\cdot) \rightarrow \Aut(A,\cdot)
\end{eqnarray*}
be group homomorphisms. The following statements are equivalent:
\begin{enumerate}
\item \label{enum_external_sdp_1} $ A \rtimes_{\theta}^{\varphi} B $ is a skew brace;
\item \label{enum_external_sdp_2} The skew brace $ B $ acts on the skew brace $ A $ via $ \varphi $ and the following equations hold inside $ \Aut(A,\cdot) $: 
\begin{eqnarray}
\gamma_{a} \theta_{b} = \theta_{b} \gamma_{a} \mbox{ for all } a \in A \mbox{ and } b \in B; \label{eqn_external_sdp_1} \\
\varphi_{b} \theta_{b'} = \theta_{b \gamma_{b}(b') b^{-1}} \varphi_{b} \mbox{ for all } b, b' \in B. \label{eqn_external_sdp_2}
\end{eqnarray}
\end{enumerate}
\end{theorem}
\begin{proof}
First suppose that \ref{enum_external_sdp_1} holds. Then we have
\begin{eqnarray*}
(e,b) \circ (a,e)(a',e) & = & (e,b) \circ (aa',e) \\
& = & (\varphi_{b}(aa'),b), 
\end{eqnarray*}
but also
\begin{eqnarray*}
(e,b) \circ (a,e)(a',e) & = & [(e,b) \circ (a,e)] (e,b)^{-1} [(e,b) \circ (a',e)] \mbox{ by \eqref{eqn_skew_brace} } \\
& = & (\varphi_{b}(a),b)(e,b^{-1}) (\varphi_{b}(a'),b) \\
& = & (\varphi_{b}(a)\varphi_{b}(a'),b).
\end{eqnarray*}
Hence $ \varphi_{b}(aa') = \varphi_{b}(a)\varphi_{b}(a') $, so $ \varphi_{b} \in \Aut(A,\cdot) $, and so the skew brace $ B $ acts on the skew brace $ A $. 

To complete the proof we assume that $ B $ acts on $ A $ via $ \varphi $ and show that the skew brace relation is satisfied on $ A \rtimes_{\theta}^{\varphi} B $ if and only if \eqref{eqn_external_sdp_1} and \eqref{eqn_external_sdp_2} hold. Let $ a,c,c' \in A $ and $ b,d,d' \in B $. Then we have
\begin{align}
(a,b) \circ (c,d)(c',d') &= (a,b) \circ (c \theta_{d}(c'), dd') \nonumber \\
&= (a \circ \varphi_{b}(c\theta_{d}(c')), b \circ dd') \nonumber \\
&= (a \circ \varphi_{b}(c)\varphi_{b}(\theta_{d}(c')), b \circ dd') \tag{ since   $\varphi_{b} \in \Aut(A,\cdot)$}\nonumber \\
&= ((a \circ \varphi_{b}(c)) a^{-1} (a \circ (\varphi_{b}(\theta_{d}(c'))), b \circ dd') \nonumber \\
&= ((a \circ \varphi_{b}(c)) \gamma_{a}(\varphi_{b}(\theta_{d}(c'))), b \circ dd') \label{eqn_external_sdp_proof_1}
\end{align}
whereas
\begin{eqnarray}
[(a,b) \circ (c,d)](a,b)^{-1}[(a,b) \circ (c',d')] 
& = & (a \circ \varphi_{b}(c), b \circ d)(\theta_{b}^{-1}(a^{-1}),b^{-1})(a \circ \varphi_{b}(c'),b \circ d') \nonumber \\
& = & (a \circ \varphi_{b}(c), b \circ d)(\theta_{b}^{-1}(a^{-1})\theta_{b}^{-1}(a \circ \varphi_{b}(c')),b^{-1}(b \circ d')) \nonumber \\
& = & (a \circ \varphi_{b}(c), b \circ d)(\theta_{b}^{-1}(a^{-1}(a \circ \varphi_{b}(c')),b^{-1}(b \circ d')) \nonumber \\
& = & (a \circ \varphi_{b}(c), b \circ d)(\theta_{b}^{-1}(\gamma_{a}(\varphi_{b}(c')),b^{-1}(b \circ d')) \nonumber \\
& = & ((a \circ \varphi_{b}(c)) \theta_{b \circ d}\theta_{b}^{-1}(\gamma_{a}(\varphi_{b}(c')),(b \circ d)b^{-1}(b \circ d')) \nonumber \\
& = & ((a \circ \varphi_{b}(c)) \theta_{b}\theta_{\gamma_{b}(d)}\theta_{b}^{-1}(\gamma_{a}(\varphi_{b}(c')),b \circ dd') \label{eqn_external_sdp_proof_2}
\end{eqnarray}
Hence $ A \rtimes_{\theta}^{\varphi} B $ is a skew brace if and only if \eqref{eqn_external_sdp_proof_1} and \eqref{eqn_external_sdp_proof_2} agree, which occurs if and only if
\begin{equation} \label{eqn_external_sdp_proof_3}
\gamma_{a}\varphi_{b}\theta_{d} = \theta_{b\gamma_{b}(d)b^{-1}}\gamma_{a}\varphi_{b} \in \Aut(A) \mbox{ for all } a \in A, b \in B.
\end{equation}
If \eqref{eqn_external_sdp_proof_3} holds then choosing $ b=1 $ yields $ \gamma_{a} \theta_{d} = \theta_{d}\gamma_{a} $, and so \eqref{eqn_external_sdp_1} holds; similarly choosing $ a=1 $ yields $ \varphi_{b}\theta_{d} = \theta_{b\gamma_{b}(d)b^{-1}}\varphi_{b} $, so \eqref{eqn_external_sdp_2} holds. Conversely, if \eqref{eqn_external_sdp_1} and \eqref{eqn_external_sdp_2} hold then \eqref{eqn_external_sdp_proof_3} holds. 

This completes the proof that \ref{enum_external_sdp_1} and \ref{enum_external_sdp_2} are equivalent.
\end{proof}

\section{Constructing and classifying Hopf-Galois structures} \label{sec_examples}

In this section we study the applications of the results of Section \ref{sec_semidirect} to Hopf-Galois theory. By the results of \cite{ST23} summarised in Section \ref{sec_skew_braces} above, the Hopf-Galois stuctures on a Galois extension $ L/K $ with Galois group $ (G,\circ) $ correspond bijectively with skew braces $ (G,\cdot,\circ) $. 

\begin{theorem} \label{thm_examples_SB_HGS}
Let $ L/K $ be a Galois extension whose Galois group $ (G,\circ) $ is the internal semidirect product of a normal subgroup $ (A,\circ) $ and a subgroup $ (B,\circ) $. Then there is a bijection between
\begin{enumerate}
\item \label{enum_HGS_sdp_1} Skew braces $ G=(G,\cdot,\circ) $ that are the internal semidirect product of an ideal $ A $ and a left ideal $ B $;
\item \label{enum_HGS_sdp_2} Hopf-Galois structures on $ L/K $ that realise $ L^{A} $ via a normal Hopf subalgebra and $ L^{B} $ via a Hopf subalgebra. 
\end{enumerate}
\end{theorem}
\begin{proof}
By the discussion in Section \ref{sec_skew_braces}, if $ G $ is a skew brace as in \ref{enum_HGS_sdp_1} then the corresponding Hopf-Galois structure realises $ L^{A} $ via a normal Hopf algebra and realises $ L^{B} $ via a Hopf subalgebra.

Conversely, if we are given a Hopf-Galois structure as in \ref{enum_HGS_sdp_2} let $ G $ be the corresponding skew brace. Then $ A $ is an ideal of $ G $ and $ B $ is a left ideal of $ G $, and the assumption that $ (G,\circ) $ is the semidirect product of $ (A,\circ) $ and $ (B,\circ) $ as groups implies that $ A \circ B = G $ and $ A \cap B = \{ e \} $, so that $ G $ is the internal semidirect product of $ A $ and $ B $ as skew braces. 
\end{proof}

Since Theorem \ref{thm_external_semidirect} classifies the skew braces appearing in Theorem \ref{thm_examples_SB_HGS} part \ref{enum_HGS_sdp_1} we may use it to classify Hopf-Galois structures on $ L/K $ appearing in Theorem \ref{thm_examples_SB_HGS} part \ref{enum_HGS_sdp_2}, as follows: let $ \varphi : (B,\circ) \rightarrow \Aut(A,\circ) $ be defined by $ \varphi_{b}(a) = b \circ a \circ \bar{b} $ for all $ a \in A $ and $ b \in B $ (recall that $ \bar{b} $ denotes the inverse of $ b $ in $ (B,\circ) $). Given skew braces $ (A,\cdot,\circ) $ and $ (B,\cdot,\circ) $ such that $ B $ acts on $ A $ via $ \varphi $, we seek to construct, or classify, all group homomorphisms $ \theta : (B,\cdot) \rightarrow \Aut(A,\cdot) $ such that \eqref{eqn_external_sdp_1} and \eqref{eqn_external_sdp_2} are satisfied; we shall say that such $ \theta $ are \textit{admissible} for the given $ A,B, $ and $ \varphi $. For each admissible $ \theta $ we obtain a Hopf-Galois structure on $ L/K $ as in Theorem \ref{thm_examples_SB_HGS} part \ref{enum_HGS_sdp_2}. 

\begin{example} \label{eg_classical}
Let $ (A,\circ,\circ) $ and $ (B,\circ,\circ) $ be trivial skew braces. Since $ \varphi_{b} \in \Aut(A,\circ) $ for each $ b \in B $ the skew brace $ B $ acts on $ A $ via $ \varphi $. Let $ \theta = \varphi $. Since $ A $ is a trivial skew brace we have $ \gamma_{a} = \mathrm{id} $ for all $ a \in A $, and so \eqref{eqn_external_sdp_1} is satisfied. Similarly, since $ B $ is a trivial skew brace we have
\[ \theta_{b\gamma_{b}(b')b^{-1}} \varphi_{b} = \varphi_{bb'b^{-1}} \varphi_{b} = \varphi_{b \circ b' \circ \bar{b} \circ b} = \varphi_{b \circ b'} = \varphi_{b} \theta_{b'} \]
for all $ b,b' \in B $; therefore \eqref{eqn_external_sdp_2} holds, and so $ \theta = \varphi $ is admissible. The skew brace we obtain in this way is the trivial skew brace on $ (G,\circ) $, which corresponds to the classical Hopf-Galois structure on $ L/K $. 
\end{example}

\begin{example}
Suppose that $ (A,\circ) = \langle a \rangle \cong C_{8} $ and $ (B,\circ) = \langle b \rangle \cong C_{4} $. Throughout this example we denote powers of $ a $ (respectively, of $ b $) with respect to $ \circ $ by $ a^{[i]} $ (respectively, $ b^{[i]} $. Suppose that $ (G,\circ) \cong (A,\circ) \rtimes_{\varphi} (B,\circ) $, where $ \varphi : (B,\circ) \rightarrow \Aut(A,\circ) $ is defined by $ \varphi_{b}(a) = a^{[-1]} $.

Define a binary operation $ \cdot $ on $ A $ by $ a^{[i]} \cdot a^{[j]} = a^{[i+(-1)^{i}j]} $; then $ (A,\cdot) \cong D_{4} $ (the element $ r=a^{[2]} $ has order $ 4 $, the element $ s=a $ has order $ 2 $, and $ s \cdot r \cdot s \cdot r = e $) and $ (A,\cdot,\circ) $ is a skew brace. Similarly, define a binary operation $ \cdot $ on $ B $ by $ b^{[i]} \cdot b^{[j]} = b^{[i+j+2ij]} $; then $ (B,\cdot) \cong C_{2} \times C_{2} $ and $ (B,\cdot,\circ) $ is a skew brace. 

We can verify that the skew brace $ B $ acts on the skew brace $ A $ via $ \varphi $: note that $ \varphi_{b}(a^{[i]})=a^{[-i]} $, so
\[ \varphi_{b}(a^{[i]} \cdot a^{[j]}) = \varphi_{b}([a^{i+(-1)^{i}j}]) = a^{[-(i+(-1)^{i}j)]} = a^{[-i]} \cdot a^{[-j]} = \varphi_{b}(a^{[i]}) \cdot \varphi_{b}(a^{[j]}). \]

Now we exhibit an admissible homomorphism $ \theta : (B,\cdot) \rightarrow \Aut(A,\cdot) $. Let $ \iota $ denote the inner automorphism of $ (A,\cdot) $ corresponding to $ r=a^{[2]} $, and let
\[ \theta_{e}=\theta_{b^{[2]}} = \mathrm{Id}, \quad \theta_{b}=\theta_{b^{[3]}} = \iota. \]
Noting that $ (B,\cdot) \cong C_{2} \times C_{2} $, we see that $ \theta $ is a homomorphism. A routine verification shows that $ \iota $ is central in $ \Aut(A,\cdot) $, so \eqref{eqn_external_sdp_1} holds, and \eqref{eqn_external_sdp_2} holds if and only if $ \theta_{\gamma_{b^{[i]}}(b^{[j]})} = \theta_{b^{[j]}} $ for all $ i $ and $ j $. We may verify that $ \gamma $-function of $ B $ is given by $ \gamma_{b^{[i]}}(b^{[j]}) = b^{[(2i+1)j]} $; comparing this with the definition of $ \theta $ we see that \eqref{eqn_external_sdp_2} is indeed satisfied, and so $ \theta $ is an admissible homomorphism. 

Thus we obtain a Hopf-Galois structure of type $ D_{4} \rtimes_{\theta} (C_{2} \times C_{2}) $ on a Galois extension with Galois group $ C_{8} \rtimes_{\varphi} C_{4} $. 
\end{example}

\begin{example} \label{eg_cyclic_abelian_all_admissible}
Suppose that $ (A,\circ) $ is cyclic group and that $ (B,\circ) $ is an abelian group, and let $ (A,\circ,\circ) $ and $ (B,\circ,\circ) $ be trivial braces. Then $ B $ acts on $ A $ via $ \varphi $, as above. Let $ \theta : (B,\circ) \rightarrow \Aut(A,\circ) $ be any group homomorphism. The fact that $ A $ is trivial implies that $ \gamma_{a} = \mathrm{id} $ for each $ a \in A $, so \eqref{eqn_external_sdp_1} is satisfied. Similarly, the fact that $ B $ is trivial and $ (B,\circ) $ is abelian implies that $ b\gamma_{b}(b')b^{-1} = b \circ b' \circ \bar{b} = b' $ for all $ b,b' \in B $; combined with the fact that $ \Aut(A,\circ) $ is abelian, this shows that \eqref{eqn_external_sdp_2} holds. Therefore all $ \theta $ are admissible in this case. 
\end{example}

\begin{example} \label{eg_powers_of_admissible}
Suppose that $ B $ acts on $ A $ via $ \varphi $, and that $ \theta $ is an admissible homomorphism. For $ i \in \mathbb{N} $ define $ \theta^{(i)} : B \rightarrow \Aut(A,\cdot) $ by $ \theta^{(i)}_{b}(a) = \theta_{b}^{i}(a) $. Then $ \theta^{(i)} $ is admissible: \eqref{eqn_external_sdp_1} is immediate, and since each $ \gamma_{b} \in \Aut(B,\cdot) $ we have
\[ \theta_{b\gamma_{b}(b')b^{-1}}^{i} = \theta_{(b\gamma_{b}(b')b^{-1})^{i}} = \theta_{b\gamma_{b}((b')^{i})b^{-1}} \mbox{ for all } b,b' \in B, \] 
from which \eqref{eqn_external_sdp_2} follows.
\end{example}

By varying the choice of complement $ (B,\circ) $ to $ (A,\circ) $ in $ (G,\circ) $ we may be able to construct numerous Hopf-Galois structures on $ L/K $ that realise $ L^{A} $ via a normal Hopf subalgebra. The \textit{classification} of such Hopf-Galois structures remains delicate, however: a given Hopf-Galois structure may realise intermediate fields corresponding to multiple such complements, and so appear multiple times in our description; on the other hand, a Hopf-Galois structure may realise $ L^{A} $ without realising $ L^{B'} $ for any complement $ (B',\circ) $ to $ (A,\circ) $, and so not be covered by our results. 

\begin{example}
The trivial skew brace $ (G,\circ,\circ) $ corresponds to the classical Hopf-Galois structure; this realises $ L^{A} $ via a normal Hopf subalgebra and realises $ L^{B'} $ for every complement $ (B',\circ) $ to $ (A,\circ) $. On the other hand, the left ideal in the almost trivial skew brace $ (G,\circ^{op},\circ) $ are precisely the normal subgroups of $ (G,\circ) $, so (assuming that $ (G,\circ) $ is nonabelian) the corresponding Hopf-Galois structure realises $ L^{A} $ via a normal Hopf subalgebra but does not realise $ L^{B'} $ for for any non-normal complement $ (B',\circ) $ to $ (A,\circ) $.
\end{example}

More generally, we have

\begin{lemma} \label{lem_opposites}
Let $ G=(G,\cdot,\circ) $ be a skew brace and $ G^{op}=(G,\cdot^{op},\circ) $ be the opposite skew brace. A subgroup $ (H,\circ) $ of $ (G,\circ) $ is a left ideal of $ G^{op} $ if and only if $ \gamma_{g}(H)=g^{-1}\cdot H \cdot g $ for all $ g \in G $ (where $ \gamma $ denotes the $ \gamma $-function of $ G $). 
\end{lemma}
\begin{proof}
The $ \gamma $-function of $ G^{op} $ is given by
\[ \gamma^{op}_{g}(h) = g^{-1} \cdot^{op} (g \circ h) = (g \circ h) \cdot g^{-1} = g \cdot \gamma_{g}(h) \cdot g^{-1} \]
for all $ g,h \in G $. Thus $ \gamma^{op}_{g}(H)=H $ for all $ g \in G $ if and only if $ \gamma_{g}(H)=g^{-1}\cdot H \cdot g $ for all $ g \in G $.
\end{proof}

We conclude this section by studying Galois extensions $ L/K $ of degree $ pq $, where $ p>q $ are prime numbers. The Hopf-Galois structures on these extensions are classified by Byott \cite{By04c}. We shall show that Theorem \ref{thm_external_semidirect}, together with the notion of opposite Hopf-Galois structures, are sufficient to describe all of these Hopf-Galois structures, although we still rely on the enumeration given in \cite{By04c}.

\begin{example}
Let $ p>q $ be prime numbers and let $ L/K $ be a Galois extension of degree $ pq $ with Galois group $ (G,\circ) $. If $ p \not \equiv 1 \pmod{q} $ then $ L/K $ is necessarily a cyclic extension and admits only the classical Hopf-Galois structure \cite[Theorem 1]{By96}, which we can describe via our methods (see Example \ref{eg_classical}). We shall therefore assume that $ p \equiv 1 \pmod{q} $. 

If $ G=(G,\cdot,\circ) $ is a skew brace then $ (G,\cdot) $ has a unique Sylow subgroup $ A $ of order $ p $; by uniqueness we see that $ A $ is a left ideal of $ G $, so $ (A,\circ) $ is also the unique Sylow $ p $ subgroup of $ (G,\circ) $, and so in fact $ A $ is an ideal of $ G $. Since $ A $ has order $ p $, and $ G/A $ has order $ q $, any left ideal of $ G $ that is a complement to $ A $ is necessarily trivial as a skew brace. 

First suppose that $ (G,\circ) $ is cyclic, and let $ (B,\circ) $ be the unique Sylow $ q $-subgroup of $ (G,\circ) $; then $ (G,\circ) \cong (A,\circ) \times (B,\circ) $. There are $ q $ homomorphisms $ \theta: (B,\circ) \rightarrow \Aut(A,\circ) \cong \Z_{p}^{\times} $, and by \ref{eg_cyclic_abelian_all_admissible} they are all admissible. Choosing $ \theta $ to be the trivial homomorphism yields the trivial skew brace, which corresponds to the classical Hopf-Galois structure on $ L/K $. Each of the other $ q-1 $ choices for $ \theta $ yields a distinct skew brace $ (G,\cdot,\circ) $ in which $ (G,\cdot) $ is metacylic, and which is the internal semidirect product of the ideal $ A $ and the left ideal $ B $. Since $ (B,\cdot) $ is not normal in $ (G,\cdot) $, Lemma \ref{lem_opposites} implies that $ B $ is not a left ideal of $ G^{op} $; thus we obtain a further $ q-1 $ skew braces $ (G,\cdot,\circ) $ in which the ideal $ A $ does not have a complement that is a left ideal. Thus in total we obtain $ 2(q-1) $ Hopf-Galois structures of metacyclic type on $ L/K $, along with the classical structure, which has cyclic type. By \cite[Theorem 6.1]{By04c} these are all the Hopf-Galois structures on $ L/K $. 

Now suppose that $ (G,\circ) $ is metacyclic, and let $ (B,\circ) $ be one of the $ p $ Sylow $ q $-subgroups of $ (G,\circ) $; then $ (G,\circ) \cong (A,\circ) \rtimes_{\varphi} (B,\circ) $ where $ \varphi_{b}(a) = b \circ a \circ \bar{b} $. As above, each of the $ q $ homomorphisms $ \theta: (B,\circ) \rightarrow \Aut(A,\circ) $ is admissible. 

Choosing $ \theta $ to be the trivial homomorphism yields a skew brace $ (G,\cdot,\circ) $ in which $ (G,\cdot) $ is cyclic, and which is the internal semidirect product of the ideal $ A $ and the left ideal $ B $. Since $ (B,\cdot)$ is the unique Sylow $ q $-subgroup of $ (G,\cdot) $, this skew brace does not have any other left ideals of order $ q $. Thus the $ p $ choices of $ (B,\circ) $ yield $ p $ different skew braces, corresponding to $ p $ different Hopf-Galois structures on $ L/K $. By \cite[Theorem 6.2]{By04c} these are all the Hopf-Galois structures of cyclic type on $ L/K $. 

Each of the nontrivial choices of $ \theta $ has the form $ \varphi^{(i)} $ for some $ i=1, \ldots ,q-1 $. Choosing $ i=1 $ yields the trivial skew brace regardless of the choice of $ (B,\circ) $, which corresponds to the classical Hopf-Galois structure; since $ (G,\circ) $ is nonabelian the opposite of the trivial skew brace (the almost trivial skew brace) corresponds to a different Hopf-Galois structure on $ L/K $. 

Each of the remaining $ q-2 $ choices for $ i $ yields a distinct skew brace $ (G,\cdot,\circ) $ in which $ (G,\cdot) $ is metacylic, and which is the internal semidirect product of the ideal $ A $ and the left ideal $ B $. In fact $ B $ is the only left ideal of $ G $ of order $ q $: the other subgroups of $ (G,\cdot) $ of order $ q $ has the form $ \langle a^{k}b \rangle $ for some $ k = 1, \ldots ,p-1 $, and we have 
\begin{eqnarray*}
\gamma_{b}(a^{k}b) & = & \gamma_{b}(a^{k}) \gamma_{b}(b) \\
& = & b^{-1} (b \circ a^{k}) b \\
& = & b^{-1}(\varphi_{b}(a^{k}) \circ b)b \\
& = & (b^{-1}\varphi_{b}(a^{k})b)b \\
& = & \varphi_{b^{1-i}}(a^{k})b \\
& \not \in & \langle a^{k}b \rangle,
\end{eqnarray*}
so none of these is a left ideal of $ G $. Hence we obtain $ q-2 $ different Hopf-Galois structures on $ L/K $. Since $ (B,\cdot) $ is not normal in $ (G,\cdot) $, Lemma \ref{lem_opposites} implies that $ B $ is not a left ideal of $ G^{op} $; thus we obtain a further $ q-2 $ skew braces $ (G,\cdot,\circ) $ in which the ideal $ A $ does not have a complement that is a left ideal. 
 
Allowing $ (B,\circ) $ to range over the $ p $ Sylow $ q $-subgroups of $ (G,\circ) $, we obtain $ 2p(q-2) $ distinct skew braces corresponding to Hopf-Galois structures on $ L/K $; together with the classical structure and its opposite we obtain $ 2 + 2p(q-2) $ Hopf-Galois structures, which by \cite[Theorem 6.2]{By04c} is all the Hopf-Galois structures of metacyclic type on $ L/K $.  
\end{example}

\section{Induced Hopf-Galois structures revisited} \label{sec_induced_HGS}

If $ A=(A,\cdot,\circ) $ and $ B=(B,\cdot,\circ) $ are skew braces and $ B $ acts on $ A $ via $ \varphi : (B,\circ) \rightarrow \Aut(A) $ then we may certainly choose $ \theta $ to be trivial in Theorem \ref{thm_external_semidirect}, giving a skew brace $ A \rtimes_{id}^{\varphi} B $ (see also \cite[Corollary 2.37]{SV18}). 

Now if $ L/K $ is a Galois extension of fields whose Galois group $ (G,\circ) $ is isomorphic to $ (A,\circ) \rtimes_{\varphi} (B,\circ) $ then $ (A,\cdot,\circ) $ corresponds to a Hopf-Galois structure of type $ (A,\cdot) $ on $ L/L^{A} $ and $ (B,\cdot,\circ) $ corresponds to a Hopf-Galois structure of type $ (B,\cdot) $ on $ L/L^{B} $; assuming that $ B $ acts on $ A $ via $ \varphi $, the skew brace $ A \rtimes_{id}^{\varphi} B $ corresponds to a Hopf-Galois structure of type $ (A,\cdot) \times (B,\cdot) $ on $ L/K $. In this section we show that this construction is equivalent to the notion of induced Hopf-Galois structures due to Crespo, Rio, and Vela \cite{CRV16a}. 

The results of \cite{CRV16a} are expressed in terms of the Greither-Pareigis classification of Hopf-Galois structure on finite separable extensions \cite{GP87}. For $ L/K $ a Galois extension with Galois group $ (G,\circ) \cong (A,\circ) \rtimes_{\varphi} (B,\circ) $, this classification implies that the Hopf-Galois structures on $ L/K $ correspond bijectively with regular subgroups of $ \Perm(G) $ that are normalised by the image of $ G $ under the left regular representation $ \lambda_{\circ}: (G,\circ) \rightarrow \Perm(G) $, and that analogous statements hold for the Galois extensions $ L/L^{A} $ and $ L/B^{B} $. It also implies that the Hopf-Galois structures on the (potentially non-normal) extension $ L^{B}/K $ correspond with regular subgroups of $ \Perm(X) $, where $ X=G/B $ is the left coset space of $ B $ in $ G $, that are normalised by the image of $ G $ under the left translation map $ \lambda_{X}: (G,\circ) \rightarrow \Perm(G/B) $. As with Hopf-Galois structures arising from skew braces, we refer to the isomorphism class of a regular subgroup as the \textit{type} of the corresponding Hopf-Galois structure.  

The classifications of Hopf-Galois structures on Galois extensions via skew braces and via regular subgroups are related as follows: if $ (G,\cdot,\circ) $ is a skew brace then $ \rho_{\cdot}(G) $ (the image of $ G $ under the right regular representation with respect to $ \cdot $) is a regular subgroup of $ \Perm(G) $ normalised by $ \lambda_{\circ}(G) $; conversely, given such a subgroup $ N $ the map $ \eta \mapsto \eta^{-1}[e_{G}] $ is a bijection, which we can use to transport the structure of $ N $ to give a second binary operation $ \cdot $ on $ G $ such that $ (G,\cdot,\circ) $ is a skew brace and $ N = \rho_{\cdot}(G) $. 

Now we summarise the theory of induced Hopf-Galois structures, essentially following \cite{CRV16a} and \cite[Section 8.3]{HAGMT}. Let $ L/K $ be a Galois extension of fields whose Galois group $ (G,\circ) $ is isomorphic to $ (A,\circ) \rtimes_{\varphi} (B,\circ) $. Suppose we are given a Hopf-Galois structure on the (potentially non-normal) extension $ L^{B}/K $, with corresponding regular subgroup $ M \leq \Perm(X) $, and a Hopf-Galois structure on the extension $ L/L^{B} $, with corresponding regular subgroup $ N \leq \Perm(B) $.

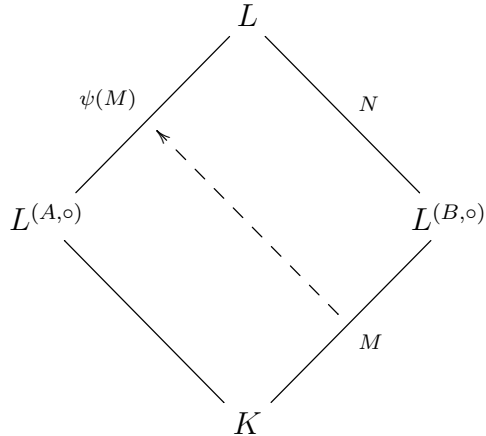
\[
\xymatrixcolsep{4pc} 
\xymatrixrowsep{5pc}
\xymatrix{
& \ar@{-}[dl]_{\psi(M)}^{}="top" L \ar@{-}[dr]^{N} & \\
 L^{(A,\circ)} \ar@{-}[dr] & & \ar@{-}[dl]^{M}_{}="bottom" L^{(B,\circ)} \\
& K & 
\ar@{-->}"bottom";"top"
} 
\]

The fact that the elements of $ A $ form a system of coset representatives for $ B $ in $ G $ implies that the map $ \psi : \Perm(X) \rightarrow \Perm(A) $ defined by 
\begin{equation} \label{eqn_defn_of_psi}
\psi(\mu)[a] \circ B = \mu[a \circ B] \mbox{ for all } a \in A 
\end{equation}
is an isomorphism of groups. It follows that $ \psi(M) $ is a regular subgroup of $ \Perm(A) $ and that the map $ \nu : M \times N \rightarrow \Perm(G) $ defined by
\begin{equation} \label{eqn_induced_HGS}
\nu(\mu,\eta)[a \circ b] = \psi(\mu)[a] \circ \eta[b]
\end{equation}
is an embedding whose image is a regular subgroup of $ \Perm(G) $. The normalisation assumptions on $ M $ and $ N $, together with the assumption that $ A $ is normal in $ G $, imply that $ \mathrm{Im}(\nu) $ is normalised by $ \lambda_{\circ}(G) $ and therefore corresponds to a Hopf-Galois structure of type $ M \times N $ on $ L/K $, which is said to be \textit{induced} from the given structures on $ L^{B}/K $ and $ L/L^{B} $. 

We begin the process of connecting this approach with ours by examining the connection between the normalisation properties of $ M $ and $ \psi(M) $ in a little more detail. 

\begin{proposition} \label{prop_induced_X_to_A}
Let $ M $ be a regular subgroup of $ \Perm(X) $. Then $ M $ is normalised by $ \lambda_{X}(G) $ if and only if the regular subgroup $ \psi(M) $ of $ \Perm(A) $ is normalised by $ \lambda_{\circ}(A) $ and $ \varphi_{b} \psi(M) \varphi_{b}^{-1} = \psi(M) $ for all $ b \in B $. 
\end{proposition}
\begin{proof}
Since $ (G,\circ) \cong (A,\circ) \rtimes_{\varphi} (B,\circ) $ the subgroup $ M $ is normalised by $ \lambda_{X}(G) $ if and only if it is normalised by $ \lambda_{X}(A) $ and $ \lambda_{X}(B) $. Recall that the inverse of an element $ g \in (G,\circ) $ is denoted by $ \bar{g} $. 

For $ \mu \in M $, $ a \in A $, and $ a' \circ B \in X $ (with $ a' \in A $) we calculate
\begin{align*}
\lambda_{X}(a) \mu \lambda_{X}(\bar{a})[ a' \circ B]
& = a \circ \mu[\bar{a} \circ a' \circ B] \\
& = (a \circ \psi(\mu)[\bar{a} \circ a']) \circ B \tag{ by the definition of $ \psi $} \\
& = (\lambda_{\circ}(a) \psi(\mu) \lambda_{\circ}(\bar{a})[a']) \circ B. 
\end{align*}
Therefore $ M $ is normalised by $ \lambda_{X}(A) $ if and only if $ \psi(M) $ is normalised by $ \lambda_{\circ}(A) $.

Similarly, for $ \mu \in M $, $ b \in B $, and $ a' \circ B \in X $ (with $ a' \in A $) we calculate 
\begin{eqnarray*}
\lambda_{X}(b) \mu \lambda_{X}(\bar{b})[ a' \circ B]
& = & b \circ \mu[\bar{b} \circ a' \circ B] \\
& = & b \circ \mu[\varphi_{\bar{b}}(a') \circ \bar{b} \circ B] \\
& = & b \circ \mu[\varphi_{b}^{-1}(a') \circ B] \\
& = & (b \circ \psi(\mu)\varphi_{b}^{-1}[a']) \circ B \\
& = & (b \circ \psi(\mu)\varphi_{b}^{-1}[a'] \circ \bar{b}) \circ B \\
& = & (\varphi_{b}\psi(\mu)\varphi_{b}^{-1}[a']) \circ B.
\end{eqnarray*}
Therefore $ M $ is normalised by $ \lambda_{X}(B) $ if and only if $ \varphi_{b} \psi(M) \varphi_{b}^{-1} = \psi(M) $ for all $ b \in B $. 
\end{proof}

Since $ \psi $ is an isomorphism that preserves regularity, Proposition \ref{prop_induced_X_to_A} also implies that if $ M' $ is a regular subgroup of $ \Perm(A) $ that is normalised by $ \lambda_{\circ}(A) $ and satisfies $ \varphi_{b} M' \varphi_{b}^{-1} = M' $ for all $ b \in B $ then $ \psi^{-1}(M') $ is a regular subgroup of $ \Perm(X) $ normalised by $ \lambda_{X}(G) $.

We have seen that regular subgroups $ M' $ of $ \Perm(A) $ normalised by $ \lambda_{\circ}(A) $ are precisely the subgroups of the form $ M'=\rho_{\cdot}(A) $, where $ \cdot $ is a binary operation on $ A $ such that $ (A,\cdot,\circ) $ is a skew brace. We can reinterpret the condition $ \varphi_{b} M' \varphi_{b}^{-1} = M' $ in terms of the corresponding binary operation on $ A $. 

\begin{proposition} \label{prop_induced_B_acts_on_A}
Let $ (A,\cdot,\circ) $ be a skew brace. Then we have $ \varphi_{b} \rho_{\cdot}(A) \varphi_{b}^{-1} = \rho_{\cdot}(A) $ for all $ b \in B $ if and only if $ \varphi_{b} \in \Aut(A,\cdot) $ for all $ b \in B $. 
\end{proposition}
\begin{proof}
Let $ a,a' \in A $ and $ b \in B $. Then we have
\begin{equation} \label{eqn_induced_B_acts_on_A}
\varphi_{b}\rho_{\cdot}(a)\varphi_{b}^{-1}[a'] = \varphi_{b}[\varphi_{b}^{-1}(a') \cdot a^{-1}].
\end{equation}

If $ \varphi_{b} \in \Aut(A,\cdot) $ then \eqref{eqn_induced_B_acts_on_A} implies immediately that 
\begin{eqnarray*}
\varphi_{b}\rho_{\cdot}(a)\varphi_{b}^{-1}[a'] & = & a' \cdot \varphi_{b}(a)^{-1} \\
& = & \rho_{\cdot}(\varphi_{b}(a))[a'],
\end{eqnarray*}
so $ \varphi_{b} \rho_{\cdot}(A) \varphi_{b}^{-1} = \rho_{\cdot}(A) $.
 
Conversely, if $ \varphi_{b} \rho_{\cdot}(A) \varphi_{b}^{-1} = \rho_{\cdot}(A) $ then choosing $ a'= e $ in \eqref{eqn_induced_B_acts_on_A} we see that we must have $ \varphi_{b} \rho_{\cdot}(a) \varphi_{b}^{-1} = \rho_{\cdot}(\varphi_{b}(a)) $. Therefore
\[ \varphi_{b} \rho_{\cdot}(a\cdot a') \varphi_{b}^{-1} = \rho_{\cdot}(\varphi_{b}(a \cdot a')), \]
but also
\begin{eqnarray*}
\varphi_{b} \rho_{\cdot}(a\cdot a') \varphi_{b}^{-1} & = & 
(\varphi_{b} \rho_{\cdot}(a) \varphi_{b}^{-1})(\varphi_{b} \rho_{\cdot}(a') \varphi_{b}^{-1}) \\
& = & \rho_{\cdot}(\varphi_{b}(a))\rho_{\cdot}(\varphi_{b}(a')) \\
& = & \rho_{\cdot}(\varphi_{b}(a) \cdot \varphi_{b}(a')), 
\end{eqnarray*}
which implies that $ \varphi_{b}(a \cdot a') = \varphi_{b}(a) \cdot \varphi_{b}(a') $. 
\end{proof}

\begin{theorem} \label{thm_induced}
Let $ L/K $ be a Galois extension of fields whose Galois group $ (G,\circ) $ is the semidirect product of a normal subgroup $ (A,\circ) $ and a subgroup $ (B,\circ) $. Suppose we are given a Hopf-Galois structure on $ L^{B}/K $, with corresponding regular subgroup $ M \leq \Perm(X) $, and a Hopf-Galois structure on the extension $ L/L^{B} $, with corresponding regular subgroup $ N \leq \Perm(B) $. Let $ (B,\cdot,\circ) $ be the skew brace corresponding to $ N \leq \Perm(B) $.

Then there is a skew brace $ (A,\cdot,\circ) $ such that $ \rho_{\cdot}(A) = \psi(M) $, the skew brace $ A \rtimes^{\varphi}_{id} B $ is defined, and the corresponding Hopf-Galois structure on $ L/K $ coincides with that induced from $ M $ and $ N $. 
\end{theorem}
\begin{proof}
We recall that the Hopf-Galois structure induced from $ M $ and $ N $ corresponds to the regular subgroup $ \nu(M \times N) \leq \Perm(G) $, where $ \nu : M \times N \rightarrow \Perm(G) $ is defined by
\[ \nu(\mu,\eta)[a \circ b] = \psi(\mu)[a] \circ \eta[b]. \]

We note that since $ (B,\cdot,\circ) $ is the skew brace corresponding to $ N \leq \Perm(B) $, we have $ \rho_{\cdot}(B)=  N $. 

Next, since $ M $ is normalised by $ \lambda_{X}(G) $, by Proposition \ref{prop_induced_X_to_A} $ \psi(M) $ is normalised by $ \lambda_{\circ}(A) $ and $ \varphi_{b} \psi(M) \varphi_{b}^{-1} = \psi(M) $ for all $ b \in B $. The first of these conclusions implies that there is a skew brace $ (A,\cdot,\circ) $ such that $ \rho_{\cdot}(A) = \psi(M) $; the second, combined with Proposition \ref{prop_induced_B_acts_on_A}, implies that $ \varphi_{b} \in \Aut(A,\cdot) $ for all $ b \in B $, so the skew brace $ B $ acts on the skew brace  $ A $ via $ \varphi $. 

Now choosing $ \theta $ to be trivial in Theorem \ref{thm_external_semidirect} we see that the skew brace $ A \rtimes^{\varphi}_{id} B $ is defined. The corresponding regular subgroup of $ \Perm(G) $ is $ \rho_{\cdot}(A \times B) $, which acts as follows
\[ \rho_{\cdot}(a',b')[a \circ b] = \rho_{\cdot}(a')[a] \circ \rho_{\cdot}(b')[b]. \]
Recalling that $ \rho_{\cdot}(A) = \psi(M) $ and $ \rho_{\cdot}(B)=  N $, we see that this subgroup coincides with the subgroup $ \nu(M \times N) $ constructed above. 
\end{proof}

\section{Smash products and integral module structure} \label{sec_smash}

We continue to study a Galois extension of fields $ L/K $ whose Galois group $ (G,\circ) $ is the semidirect product of a normal subgroup $ (A,\circ) $ and a subgroup $ (B,\circ) $. 

In \cite{GR20} Gil-Mu\~{n}oz and Rio show that if $ H $ is a Hopf algebra giving a Hopf-Galois structure on $ L/K $ which is induced from Hopf-Galois structures on $ L/L^{B} $ and $ L^{B}/K $ then $ H \cong  H_{A} \otimes_{K} H_{B} $, where $ H_{A} $ and $ H_{B} $ are Hopf subalgebras such that $ L^{H_{A}} = L^{A} $ and $ L^{H_{B}} = L^{B} $. They also use this description  to study the structure of rings of algebraic integers in extensions of local or global fields. In this section we generalise these results to Hopf-Galois structures arising via Theorem \ref{thm_examples_SB_HGS}, using the language of skew braces.  

We suppose that $ G=(G,\cdot,\circ) $ is a skew brace which is the semidirect product of an ideal $ A $ and a left ideal $ B $. As described in Section \ref{sec_skew_braces}, the Hopf algebra giving the corresponding Hopf-Galois structure on $ L/K $ is $ H_{G} = L[G,\cdot]^{(G,\circ)} $, where $ (G,\circ) $ acts on $ L[G,\cdot] $ via \eqref{eqn_skew_brace_HGS_action}, and $ H_{A} = L[A,\cdot]^{(G,\circ)} $ and $ H_{B} = L[B,\cdot]^{(G,\circ)} $ are Hopf subalgebras of $ H_{G} $.

The fact that the group $ (G,\cdot) $ is the semidirect product of a normal subgroup $ (A,\cdot) $ and a subgroup $ (B,\cdot) $ implies that $ L[A,\cdot] $ and $ L[B,\cdot] $ are $ L $-Hopf subalgebras of $ L[G,\cdot] $ and the action of $ (B,\cdot) $ on $ (A,\cdot) $ by automorphisms extends to an $ L $-linear action of $ L[B,\cdot] $ on $ L[A,\cdot] $:
\begin{equation}
\left( \sum_{b \in B} z_{b} b \right) \bullet \left( \sum_{a \in A} w_{a} a \right) = \sum_{a \in A} \sum_{b \in B} w_{a} z_{b} bab^{-1}.
\end{equation}
This action respects the $ L $-Hopf algebra structure maps on $ L[A,\cdot] $, so the \textit{smash product} $ L[A,\cdot] \; \#_{L} \; L[B,\cdot] $ is defined and is an $ L $-Hopf algebra. It can be shown that the natural map $ L[A,\cdot] \; \#_{L} \; L[B,\cdot] \rightarrow L[G,\cdot] $ is an isomorphism of $ L $-Hopf algebras. Our first result is that this smash product description descends to $ H_{G} $.

\begin{proposition} \label{prop_H_G_smash_product}
The smash product $ H_{A} \; \#_{K} \; H_{B} $ is defined, is a $ K $-Hopf algebra, and is isomorphic to $ H_{G} $. 
\end{proposition}
\begin{proof}
First we show that the action $ \bullet $ of $ L[B,\cdot] $ on $ L[A,\cdot] $  descends to an action of $ H_{B} $ on $ H_{A} $. Let $ w= \sum_{a} w_{a} a \in L[A,\cdot] $ and $ z = \sum_{b} z_{b} b \in L[B,\cdot] $ (with $ w_{a},z_{b} \in L $), and consider the action of $ g \in G $ via \eqref{eqn_skew_brace_HGS_action}:
\begin{eqnarray*}
g \star ( z \bullet w ) & = & g \star \left( \left(\sum_{b} z_{b} b \right) \bullet \left( \sum_{a} w_{a} a \right) \right) \\
& = & g \star \left( \sum_{a,b} z_{b} w_{a} bab^{-1} \right) \\
& = & \sum_{a,b} g(z_{b} w_{a}) \gamma_{g}(bab^{-1}) \\
& = & \sum_{a,b} g(z_{b}) g(w_{a}) \gamma_{g}(b)\gamma_{g}(a)\gamma_{g}(b)^{-1} \\ 
& = & \left(\sum_{b} g(z_{b}) \gamma_{g}(b) \right) \bullet \left( \sum_{a} g(w_{a}) \gamma_{g}(a) \right) \\
& = & (g \star z) \bullet (g \star w).
\end{eqnarray*}
Hence if $ z \in H_{B} $ and $ w \in H_{A} $ then $ z \bullet w \in H_{A} $, and so the action of $ L[B,\cdot] $ on $ L[A,\cdot] $ descends to an action of $ H_{B} $ on $ H_{A} $. This action respects the $ K $-Hopf algebra structure maps of $ H_{A} $, so the smash product $ H_{A} \; \#_{K} \; H_{B} $ is defined and is a $ K $-Hopf algebra. By Galois descent we have $ L \otimes_{K} H_{A} \cong L[A,\cdot] $ and $ L \otimes_{K} H_{B} \cong L[B,\cdot] $ as $ L $-Hopf algebras, and so
\[ L \otimes_{K} (H_{A} \; \#_{K} \; H_{B}) \cong L[A,\cdot] \; \#_{L} \; L[B,\cdot] \cong L[G,\cdot] \cong L \otimes_{K} H_{G} \]
as $ L $-Hopf algebras. Finally, since the natural map $ L[A,\cdot] \; \#_{L} \; L[B,\cdot] \rightarrow L[G,\cdot] $ is an isomorphism of $ L $-Hopf algebras that respects the actions of $ (G,\circ) $, it descends to an isomorphism of $ K $-algebras $ H_{A} \; \#_{K} \; H_{B} \cong H_{G} $. 
\end{proof}

The fact that $ (G,\circ) $ is the semidirect product of $ (A,\circ) $ and $ (B,\circ) $ also implies that the extensions $ L^{A}/K $ and $ L^{B}/K $ are linearly disjoint, so the natural map $ L^{A} \otimes_{K} L^{B} \rightarrow L^{A}L^{B} = L $ is an isomorphism of $ L^{A} $-algebras and $ L^{B} $-algebras. In the remainder of this section we relate properties of the Hopf-Galois structure given by $ H_{G} $ on $ L/K $with properties of certain Hopf-Galois structures on some of the extensions diagrammed below. 

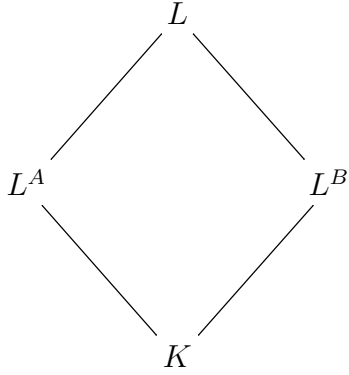
\[
\xymatrixcolsep{3pc} 
\xymatrixrowsep{4pc}
\xymatrix{
& \ar@{-}[dl] L \ar@{-}[dr] & \\
 L^{A} \ar@{-}[dr] & & \ar@{-}[dl] L^{B} \\
& K & \\
} 
\]

We have $ L^{H_{A}} = L^{A} $ as described in Section \ref{sec_skew_braces}, and so (as discussed in Section \ref{sec_Introduction}) $ L^{A} \otimes_{K} H_{A} $ gives a Hopf-Galois structure on $ L/L^{A} $. Recalling that $ H_{A}=L[A,\cdot]^{(G,\circ)} $, the theory of Galois descent implies that the natural  map $ L^{A} \otimes_{K} H_{A} \rightarrow L[A,\cdot]^{(A,\circ)} $ is an isomorphism of $ L^{A} $-Hopf algebras; it follows that the Hopf-Galois structure on $ L/L^{A} $ given by $ L^{A} \otimes_{K} H_{A} $ coincides with the one arising from the skew brace $ (A,\cdot,\circ) $. Similarly, we have $ L^{H_{B}}=L^{B} $ and the Hopf-Galois structure on $ L/L^{B} $ given by $ L^{B} \otimes_{K} H_{B} $ coincides with the one arising from the skew brace $ (B,\cdot,\circ) $.

Our next aim is to relate these Hopf-Galois structures with certain Hopf-Galois structures on the extensions $ L^{A}/K $ and $ L^{B}/K $. The case of the extension $ L^{A}/K $ is straightforward. 

\begin{proposition} \label{prop_H_B_HGS}
The $ K $-Hopf algebra $ H_{B} $ gives a Hopf-Galois structure on the extension $ L^{A}/K $. 
\end{proposition}
\begin{proof}
The fact that $ (G,\cdot) $ is the semidirect product of the normal subgroup $ (A,\cdot) $ and the subgroup $ (B,\cdot) $ implies that $ L[A,\cdot] $ is a normal Hopf subalgebra of $ L[G,\cdot] $ and there is a short exact sequence of $ L $-Hopf algebras
\[ L \rightarrow L[A,\cdot] \rightarrow L[G,\cdot] \rightarrow L[B,\cdot] \rightarrow L. \]
Since in addition $ A $ is an ideal, and $ B $ a left ideal, of the skew brace $ G $, this short exact sequence descends to a short exact sequence of $ K $-Hopf algebras (see \cite[Example 2.1]{ST23})
\[ K \rightarrow L[A,\cdot]^{(G,\circ)} \rightarrow L[G,\cdot]^{(G,\circ)} \rightarrow L[B,\cdot]^{(G,\circ)} \rightarrow K; \]
that is
\[ K \rightarrow H_{A} \rightarrow H_{G} \rightarrow H_{B} \rightarrow K. \]
Using this short exact sequence, the action of $ H_{G} $ on $ L $ (giving a Hopf-Galois structure on $ L/K $) yields an action of $ H_{B} $ on $ L^{A} $ (giving a Hopf-Galois structure on $ L^{A}/K $) \cite[Lemma 4.1]{By02}. But this action is simply the action of $ H_{B} $, viewed as a Hopf subalgebra of $ H_{G} $. Hence $ H_{B} $ gives a Hopf-Galois structure on $ L^{A}/K $. 
\end{proof}

Combining Proposition \ref{prop_H_B_HGS} with the isomorphism $ H_{G} \cong H_{A} \; \#_{K} \; H_{B} $ obtained in Proposition \ref{prop_H_G_smash_product} and the isomorphism $ L^{A} \otimes_{K} L^{B} \cong L $ reveals useful information about the action of $ H_{G} $ on $ L $. If $ w \in H_{A} $ and $ z \in H_{B} $ then for all $ \alpha \in L^{A} $ and $ \beta \in L^{B} $ we have
\begin{eqnarray}
(wz)[\alpha\beta] & = & w[ \beta z[\alpha] ] \mbox{ since $ \beta \in L^{B} $ } \nonumber \\ 
& = & w[\beta] z[\alpha] \mbox{ since $ z[\alpha] \in L^{A} $ by Proposition \ref{prop_H_B_HGS}. } \label{eqn_H_G_action_on_product}
\end{eqnarray}
It is important to note, however, that $ w[\beta] $ need not be an element of $ L^{B} $, since $ H_{A} $ need not give a Hopf-Galois structure on $ L^{B}/K $. A sufficient condition for this to occur is that $ H_{B} $ is a normal Hopf subalgebra of $ H_{G} $ (if and only if $ B $ is a strong left ideal of $ G $, if and only if $ (G,\cdot) \cong (A,\cdot) \times (B,\cdot) $); in this case we may argue as in Proposition \ref{prop_H_B_HGS} (see \cite[Section 8.3]{HAGMT}). More generally, we shall see that a Hopf algebra closely related to $ H_{A} $ \textit{does} gives a Hopf-Galois structure on $ L^{B}/K $. 

To describe this Hopf-Galois structure we once again employ the Greither-Pareigis classification, using the notation established in Section \ref{sec_induced_HGS}. In particular, we write $ X $ for the left coset space of $ (B,\circ) $ in $ (G,\circ) $ and recall that there is an isomorphism of groups $ \psi : \Perm(X) \rightarrow \Perm(A) $ defined by 
\[ \psi(\mu)[a] \circ B = \mu[a \circ B] \mbox{ for all } a \in A. \] 
Using this, we define a homomorphism $ \overline{\rho_{\cdot}} : (A,\cdot) \rightarrow \Perm(X) $ by $ \overline{\rho_{\cdot}} = \psi^{-1} \rho_{\cdot} $. Explicitly, we have
\[ \overline{\rho_{\cdot}}(a)[a' \circ B] = (a' a^{-1}) \circ B \mbox{ for all } a,a' \in A. \]

\begin{lemma} \label{lem_A_bar_regular}
The subgroup $ \overline{A} = \overline{\rho_{\cdot}}(A) $ of $ \Perm(X) $ is regular and normalised by $ \lambda_{X}(G) $.
\end{lemma}
\begin{proof}
We have $ |X|=|A| $ because the elements of $ A $ form a system of coset representatives for $ (B,\circ) $ in $ (G,\circ) $, and we have $ |\overline{A}| = |A| $ because $ \overline{\rho_{\cdot}} $ is the composition of the injective map $ \rho_{\cdot} $ and the isomorphism $ \psi^{-1} $; hence $ |\overline{A}| = |X| $. In addition, $ \overline{A} $ acts transitively on $ X $, since for all $ a \in A $ we have
\[ \overline{\rho_{\cdot}}(a^{-1})[e \circ B] = a \circ B. \]
Hence $ \overline{A} $ is a regular subgroup of $ \Perm(X) $. To show that it is normalised by $ \lambda_{X}(G) $, let $ a \in A $ and $ g \in G $; then for all $ a' \circ B \in X $ we have
\begin{eqnarray*}
\lambda_{X}(g) \overline{\rho_{\cdot}}(a) \lambda_{X}(\bar{g}) [a' \circ B] 
& = & g \circ \overline{\rho_{\cdot}}(a) [\bar{g} \circ a' \circ B] \\
& = & (g \circ (\bar{g} \circ a')a^{-1}) \circ B \\
& = & a'g^{-1}(g \circ a^{-1}) \circ B \mbox{ by \eqref{eqn_skew_brace}} \\
& = & a' \gamma_{g}(a^{-1}) \circ B \\
& = & \overline{\rho_{\cdot}}(\gamma_{g}(a)) [a' \circ B].
\end{eqnarray*}
Hence $ \overline{A} $ is normalised by $ \lambda_{X}(G) $, as claimed.
\end{proof}

Lemma \ref{lem_A_bar_regular} implies that $ \overline{A} $ corresponds to a Hopf-Galois structure on $ L^{B}/K $. The process of constructing this Hopf-Galois structure (detailed in \cite{GP87}) is very similar to the process for constructing the Hopf-Galois structure on a Galois extension corresponding to a skew brace, as described in Section \ref{sec_skew_braces}. The fact that $ \overline{A} $ is normalised by $ \lambda_{X}(G) $ implies that the group $ (G,\circ) $ acts semilinearly on the $ L $-Hopf algebra $ L[\overline{A}] $ by the rule
\begin{eqnarray*} 
g \star \sum_{a \in A} c_{a} \overline{\rho_{\cdot}}(a)
& = & \sum_{a \in A} g(c_{a}) \lambda_{X}(g) \overline{\rho_{\cdot}}(a) \lambda_{X}(g)^{-1} \\
& = & \sum_{a \in A} g(c_{a}) \overline{\rho_{\cdot}}(\gamma_{g}(a)) 
\end{eqnarray*}
and the fixed ring $ H_{\overline{A}} = L[\overline{A}]^{(G,\circ)} $ is a $ K $-Hopf algebra, which acts on $ L^{B} $ via the rule
\begin{eqnarray} 
\left( \sum_{a \in A} c_{a} \overline{\rho_{\cdot}}(a) \right) [t] 
& = & \sum_{a \in A} c_{a} \overline{\rho_{\cdot}}(a)^{-1}(e \circ B)[t] \nonumber \\
& = & \sum_{a \in A} c_{a} (a \circ B)[t] \nonumber \\
& = & \sum_{a \in A} c_{a} a[t] \label{eqn_GP_action_on_L}
\end{eqnarray}
to give a Hopf-Galois structure on $ L^{B}/K $.

To further clarify the relationship between $ H_{\overline{A}} $ and $ H_{A} $, we have the following. 

\begin{proposition} \label{prop_H_A_vs_H_A_bar}
The $ L^{A} $-Hopf algebra $ L^{A} \otimes_{K} H_{\overline{A}} $ gives a Hopf-Galois structure on $ L/L^{A} $, which is isomorphic to the Hopf-Galois structure given by $ L^{A} \otimes_{K} H_{A} $.
\end{proposition}
\begin{proof}
Recall that the natural map $ L^{A} \otimes_{K} L^{B} \rightarrow L^{A}L^{B} = L $ is an isomorphism of $ K $-algebras. By \cite[(2.11)]{TWE}, the fact that the $ K $-Hopf algebra $ H_{\overline{A}} $ gives a Hopf-Galois structure on $ L^{B}/K $ implies that the $ L^{A} $-Hopf algebra $ L^{A} \otimes_{K} H_{\overline{A}} $ gives a Hopf-Galois structure on $ (L^{A} \otimes_{K} L^{B}) / L^{A} $; that is, on $ L/L^{A} $. To show that this Hopf-Galois structure coincides with that given by $ L^{A} \otimes_{K} H_{A} $, we show that there is an isomorphism of $ L^{A} $-Hopf algebras $ \phi : L^{A} \otimes_{K} H_{A} \rightarrow L^{A} \otimes_{K} H_{\overline{A}}  $ such that $ \phi(z)[t] = z[t] $ for all $ z \in L^{A} \otimes_{K} H_{A} $ and all $ t \in L $.

We have seen that we may identify the $ L^{A} $-Hopf algebra $ L^{A} \otimes_{K} H_{A} $ with $ L[A,\cdot]^{(A,\circ)} $; a similar argument allows us to identify $ L^{A} \otimes_{K} H_{\overline{A}} $ with $ L[\overline{A}]^{(A,\circ)} $. The $ L^{A} $-Hopf algebras $ L[A,\cdot]^{(A,\circ)} $ and $ L[\overline{A}]^{(A,\circ)} $ are isomorphic if and only if there is an isomorphism of groups $ (A,\cdot) \rightarrow \overline{A} $ that respects the actions of $ (A,\circ) $ on each side. The details of the proof of Lemma \ref{lem_A_bar_regular} show that the map $ \overline{\rho_{\cdot}} $ has this property; it therefore extends to an isomorphism of $ L^{A} $-Hopf algebras $ \phi : L[A,\cdot]^{(A,\circ)} \rightarrow L[\overline{A}]^{(A,\circ)} $. 

Finally, we show that $ \phi(z)[t] = z[t] $ for all $ z \in L[A,\cdot]^{(A,\circ)} $ and all $ t \in L $. Since the Hopf algebras act $ L^{A} $-linearly and  $ L^{A}/K $ and $ L^{B}/K $ are linearly disjoint, it is sufficient to consider $ t \in L^{B} $. But for these elements the equality follows immediately from \eqref{eqn_GP_action_on_L}. Hence the $ L^{A} $-Hopf algebras 
$ L^{A} \otimes_{K} H_{\overline{A}} $ and $ L^{A} \otimes_{K} H_{A} $ give isomorphic Hopf-Galois structures on $ L/L^{A} $.  
\end{proof}

As a first example of how the properties of the Hopf-Galois structures given by $ H_{G} $ on $ L/K $, by $ H_{B} $ on $ L^{A}/K $, and by $ H_{\overline{A}} $ on $ L^{B}/K $ are connected to one another, we consider \textit{generalised normal basis generators}. The Hopf-Galois analogue of the classical normal basis theorem \cite[(2.16)]{TWE} states that $ L $ is a free $ H_{G} $-module of rank one; it is often desirable to find explicit generators. The description of $ H_{G} $ obtained in Proposition \ref{prop_H_G_smash_product} provides a method to do this. 

\begin{proposition} \label{prop_NBG}
Let $ \alpha \in L^{A} $ be a free generator of $ L^{A} $ as an $ H_{B} $-module and $ \beta \in L^{B} $ be a free generator of $ L^{B} $ as an $ H_{\overline{A}} $-module. Then
\begin{enumerate}
\item $ \beta $ is a free generator of $ L $ as an $ L^{A} \otimes_{K} H_{A} $-module; \label{enum_NBG_1}
\item $ \alpha\beta $ is a free generator of $ L $ as an $ H_{G} $-module. \label{enum_NBG_2}
\end{enumerate}
\end{proposition}
\begin{proof}
To prove \ref{enum_NBG_1} we begin with the fact that $ H_{\overline{A}} [\beta] = L^{B} $; this implies that $ L^{A} \otimes_{K} H_{\overline{A}} [\beta] = L^{A} \otimes_{K} L^{B} $, and identifying $ L^{A} \otimes_{K} L^{B} $ with $ L $ yields $ (L^{A} \otimes_{K} H_{\overline{A}}) [\beta] = L $. Finally, by Proposition \ref{prop_H_A_vs_H_A_bar} we have $ (L^{A} \otimes_{K} H_{A}) [\beta] = L $, as claimed.

To prove \ref{enum_NBG_2} we use the description of $ H_{G} $ as a smash product (Proposition \ref{prop_H_G_smash_product}) along with the properties of the action of this smash product on elements of $ L $ given in \eqref{eqn_H_G_action_on_product}. We have

\begin{align*}
H_{G}[\alpha\beta] & = (H_{A} \; \# \; H_{B})[\alpha\beta] \\
& = H_{A}[\beta] H_{B}[\alpha] \\
& = L^{A} H_{A}[\beta] \\
& = (L^{A} \otimes_{K} H_{A})[\beta] \\
& = L 
\end{align*}
by part \ref{enum_NBG_1}. This establishes \ref{enum_NBG_2}.
\end{proof}

Finally, we suppose that $ L/K $ is an extension of local fields and turn to the study of integral module structure, which has been a fruitful application of Hopf-Galois theory; we refer the reader to \cite{HAGMT} or \cite{TWE} for detailed surveys. For each intermediate field of the extension $ L/K $, we write $ \O_{F} $ for the ring of integers (valuation ring) of $ F $. If $ H $ is a Hopf algebra giving a Hopf-Galois structure on $ L/K $ then we can study the structure of $ \O_{L} $ as a module over its \textit{associated order} in $ H $:
\[ \A_{H} = \{ h \in H \mid h(\O_{L}) \subseteq \O_{L} \}, \]
with particular emphasis on the question of whether it is free (necessarily of rank one). In fact, $ \A_{H} $ is the only order in $ H $ over which $ \O_{L} $ can be free \cite[(12.5)]{TWE}. 

Continuing the theme of this section, we will relate the structure of $ \O_{L} $ over its associated order $ \A_{G} $ in $ H_{G} $ with the structure of $ \O_{L^{A}} $ over its associated order $ \A_{B} $ in $ H_{B} $ and the structure of $ \O_{L^{B}} $ over its associated order $ \A_{\overline{A}} $ in $ H_{\overline{A}} $. Our result generalises \cite[Theorem 5.16]{GR20}, which addresses this question in the case of induced Hopf-Galois structures.  

\begin{proposition}
Let $ L/K $ be an extension of local fields. Suppose that $ L^{A}/K $ is unramified and that $ \O_{L^{B}} $ is a free $ \A_{\overline{A}} $-module. Then
\begin{enumerate}
\item $ \O_{L} $ is a free $ \A_{L^{A} \otimes H_{A}} $-module; \label{enum_integral_1}
\item $ \O_{L} $ is a free $ \A_{G} $-module.
\label{enum_integral_2}
\end{enumerate}
\end{proposition}
\begin{proof}
The assumption that $ L^{A}/K $ is unramified implies that the natural map $ \O_{L^{A}} \otimes_{\O_{K}} \O_{L^{B}} \rightarrow \O_{L} $ is an isomorphism of $ \O_{L^{A}} $-algebras and $ \O_{L^{B}} $-algebras \cite[(2.13)]{FT91}. 

Let $ \beta \in \O_{L^{B}} $ be a free generator of $ \O_{L^{B}} $ as an $ \A_{\overline{A}} $-module. Then (similarly to the proof of Proposition \ref{prop_NBG} part \ref{enum_NBG_1}) we have $ \O_{L^{A}} \otimes_{\O_{K}} \O_{L^{B}} = \O_{L^{A}} \otimes_{\O_{K}} \A_{\overline{A}} [\beta] $, so (making natural identifications) $ \O_{L} = (\O_{L^{A}} \otimes_{\O_{K}} \A_{\overline{A}}) [\beta] $. Comparing ranks over $ \O_{K} $, we see that this action is free. We can view $ \O_{L^{A}} \otimes_{\O_{K}} \A_{\overline{A}} $ as an $ \O_{L^{A}} $-order in $ L^{A} \otimes_{K} H_{\overline{A}} $, and by Proposition \ref{prop_H_A_vs_H_A_bar} the Hopf-Galois structure this gives on $ L/L^{A} $ is isomorphic to that given by $ L^{A} \otimes_{K} H_{A} $. Thus we see that $ \O_{L} $ is free over an $ \O_{L^{A}} $-order in $ L^{A} \otimes H_{A} $, which therefore coincides with its associated order $ \A_{L^{A} \otimes H_{A}} $ in $ L^{A} \otimes_{K} H_{A} $. This establishes \ref{enum_integral_1}. 

We begin the proof of \ref{enum_integral_2} with two further consequences of $ L^{A}/K $ being unramified. The first is that $ \O_{L^{A}} $ is a free $ \A_{B} $-module \cite[Theorem 3.4]{Tr11}; we let $ \alpha \in \O_{L^{A}} $ be a free generator. The second is that $ \O_{L^{A}} / \O_{K} $ is a Galois extension of commutative rings with group $ (B,\circ) $ in the sense of \cite{CHR65}, so (by Galois descent at integral level) there exists a unique $ \O_{K} $-order $ \Gamma_{A} $ in $ H_{A} $ such that $ \A_{L^{A} \otimes H_{A}} = \O_{L^{A}} \otimes_{\O_{K}} \Gamma_{A} $. 

Since $ H_{G} \cong H_{A} \; \# \; H_{B} $, the $ \O_{K} $-module $ \Gamma_{A}\A_{B} $ is an $ \O_{K} $-lattice in $ H_{G} $. On one hand we have
\[ (\Gamma_{A}\A_{B})[\O_{L}] = (\Gamma_{A}\A_{B})[\O_{L^{B}}\O_{L^{A}}] = (\Gamma_{A}[\O_{L^{B}}]\A_{B}[\O_{L^{A}}] \subseteq \O_{L^{B}}\O_{L^{A}} = \O_{L} \]
by \eqref{eqn_H_G_action_on_product}, so $ \O_{L^{B}}\O_{L^{A}} \subseteq \A_{G} $. On the other hand, \eqref{eqn_H_G_action_on_product} also implies that
\begin{align*}
(\Gamma_{A}\A_{B})[\alpha\beta] & = \A_{B}[\alpha]\Gamma_{A}[\beta] \\
& = \O_{L^{A}}\Gamma_{A}[\beta] \\
& = (\O_{L^{A}} \otimes_{\O_{K}} \Gamma_{A})[\beta] \\
& = \A_{L^{A} \otimes H_{A}}[\beta] \\
& = \O_{L}.
\end{align*}
Hence
\[ \O_{L} = (\Gamma_{A}\A_{B})[\alpha\beta] \subseteq \A_{G}(\alpha\beta) \subseteq \A_{G}(\O_{L}) \subseteq \O_{L}, \]
so $ \Gamma_{A}\A_{B} = \A_{G} $ and $ \O_{L} $ is a free $ \A_{G} $-module. 
\end{proof}

\bibliographystyle{plain}

\end{document}